\documentclass[11pt,a4paper,oneside]{article}

\usepackage{amsthm}
\usepackage{amssymb}
\usepackage{graphicx}
\usepackage{titlesec}
\usepackage{colortbl}
\usepackage{subfigure}
\newcommand{\R}{\mathbb{R}}

\titleformat{\chapter}[hang]
{\color[rgb]{0.5,0.5,0.5}\bfseries\sffamily\LARGE}{\thechapter~~}{0pt}{}

\titleformat{\section}[hang]
{\bfseries\sffamily\Large}{\thesection~~}{0pt}{}

\titleformat{\subsection}[hang]
{\bfseries\sffamily}{\thesubsection~~}{0pt}{}

\titleformat{\subsubsection}[hang]
{\bfseries\sffamily\small}{}{0pt}{}

\titleformat{\paragraph}[runin]
{\bfseries\sffamily\small}{}{0pt}{}

\newtheoremstyle{thm}{16pt}{16pt}{\it}{0pt}{\bfseries\sffamily}{.~}{ }{\thmname{#1}\thmnumber{~#2}\thmnote{~(#3)}}
\theoremstyle{thm}
\newtheorem{thm}{Theorem}

\newtheoremstyle{lmm}{16pt}{16pt}{\it}{0pt}{\bfseries\sffamily}{.}{ }{\thmname{#1}\thmnumber{~#2}\thmnote{~(#3)}}
\theoremstyle{lmm}

\newtheoremstyle{crl}{16pt}{16pt}{\it}{0pt}{\bfseries\sffamily}{.}{ }{\thmname{#1}\thmnumber{~#2}\thmnote{~(#3)}}
\theoremstyle{crl}

\newtheoremstyle{prf}{2pt}{32pt}{}{0pt}{\it}{.}{ }{}
\theoremstyle{prf}

\newtheoremstyle{defi}{16pt}{16pt}{}{0pt}{\bfseries\sffamily}{.~}{ }{\thmname{#1}\thmnumber{~#2}\thmnote{~(#3)}}
\theoremstyle{defi}

\newtheoremstyle{nota}{16pt}{16pt}{}{0pt}{\bfseries\sffamily}{.~}{ }{\thmname{#1}\thmnumber{~#2}\thmnote{~(#3)}}
\theoremstyle{nota}

\newtheoremstyle{exer}{16pt}{16pt}{}{0pt}{\bfseries\sffamily}{.}{ }{\thmname{#1}\thmnumber{~#2}\thmnote{~(#3)}}
\theoremstyle{exer}

\newtheoremstyle{sol}{2pt}{32pt}{\footnotesize}{0pt}{\bfseries\sffamily}{.}{ }{}
\theoremstyle{sol}

\newtheoremstyle{exam}{16pt}{16pt}{}{0pt}{\bfseries\sffamily}{.}{ }{\thmname{#1}\thmnumber{~#2}\thmnote{~(#3)}}
\theoremstyle{exam}

\newtheoremstyle{rema}{16pt}{16pt}{}{0pt}{\bfseries\sffamily}{.}{ }{\thmname{#1}\thmnumber{~#2}\thmnote{~(#3)}}
\theoremstyle{rema}
\newtheorem{rema}{Remark}

\date{\today}

\title{\LARGE{\textsf{\textbf{Embedded Zassenhaus Expansion to Operator Splitting Schemes: Theory and Application in Fluid Dynamics}}}}

\author{
J\"urgen Geiser
\thanks{\footnotesize{{\ttfamily geiser@mathematik.hu-berlin.de}}} \\[6pt]
\footnotesize{Institute of Physics,} \\
\footnotesize{Felix-Hausdorff-Str. 6,} \\
\footnotesize{D-17489 Greifswald, Germany}}

\begin{document}

%%%%%%%%%%%%%%%%%%%%%%%%%%%%%%%%%%%%%%%%%
%%            TITLE PAGE               %%
%%%%%%%%%%%%%%%%%%%%%%%%%%%%%%%%%%%%%%%%%

\maketitle

\begin{abstract}
In this paper, we contribute operator-splitting
methods improved by the Zassenhaus product for the numerical solution of
linear partial differential equations. 
We address iterative splitting methods, that can be improved by means of the 
Zassenhaus product formula, which is a sequnential splitting scheme.
The coupling of iterative and sequential splitting techniques
are discussed and can be combined with respect to their
compuational time. While the iterative splitting schemes are cheap
to compute, the Zassenhaus product formula is more expensive, based
on the commutators but achieves higher order accuracy.
Iterative splitting schemes and also Zassenhaus products are 
applied in physics and physical chemistry are important and are 
predestinated to their combinations of each benefits.
Here we consider phase models in CFD (computational fluid dynamics).
We present an underlying analysis for obtaining higher order 
operator-splitting methods based on the Zassenhaus product.

Computational benefits are given with sparse matrices, which arose of 
spatial discretization of the underlying partial differential equations.
While Zassenhaus formula allows higher accuracy, due to the fact that 
we obtain higher order commutators, we combine such an improved 
initialization process to cheap computable to linear convergent 
iterative splitting schemes.

 Theoretical discussion about convergence and application examples are
discussed with CFD problems.

\end{abstract}

{\bf Keywords.} Operator splitting method, Iterative splitting, Zassenhaus product, Parabolic differential equations, Convection-Diffusion-Reaction Equations.

{\bf AMS subject classifications.} 80A20, 80M25, 74S10, 76R50, 35J60, 35J65, 65M99, 65Z05, 65N12.

%%%%%%%%%%%%%%%%%%%%%%%%%%%%%%%%%%%%%%%%%
%%            MAIN BODY                %%
%%%%%%%%%%%%%%%%%%%%%%%%%%%%%%%%%%%%%%%%%

\section{Introduction}

Our motivation to study the operator splitting methods
come from models in fluid dynamics problems, for
example problems in bio-remediation \cite{ewing02} or radioactive contaminants
\cite{gei01}, \cite{fro00}.

While standard splitting methods deal with lower order convergence,
we propose to a combination of iterative splitting methods with embedded Zassenhaus product formula.

Theoretically, we combine fix-point schemes (iterative splitting methods)
with sequential splitting schemes (Zassenhaus products), which 
are connected to the theory of Lie groups and Lie algebras. 
Based on that relation, we can construct higher order splitting schemes 
for an underlying Lie algebra and improve the convergence results with
cheap iterative schemes.

Historically, the efficiency of decoupling different physical processes
into more simpler processes, e.g.,  convection and reaction, helps to 
accelerate the solver process and is discussed in \cite{mclach02}.

We propose the following ideas:
\begin{itemize}
\item Iterative splitting schemes are based on fix-point schemes, e.g. Waveform relaxation, which linearly improve the convergence order. Based on reducing 
integral operators to cheap computable matrices, they can be seen as solver methods, see \cite{geiser_2011_1}.
\item Zassenhaus formula are based on nested commutators, which are main keys to
derive higher order standard splitting schemes (e.g. Lie-Trotter and 
Strang splitting). They  are simple to compute with their nil-potent structure, see \cite{geiser_3}.
\end{itemize}

In this paper we study the following mathematical equations:
The equations are coupled with the reaction terms and are presented as follows.
\begin{eqnarray}
\label{eq1}
  && \partial_t R_i u_i + \nabla \cdot {\bf v} \; u_i = - \lambda_i \; R_i \;
  u_i + \lambda_{i-1} \; R_{i-1} \; u_{i-1} \\
&& + \beta (- u_i + g_i )\; \; \mbox{in} \;  \Omega \times
  (0,T) \; , \nonumber  \\
  && u_{i,0}(x) = u_i(x,0) \; \; \mbox{on} \; \Omega \; , \\
  && \partial_t R_i g_i = - \lambda_i \; R_i \;
  g_i + \lambda_{i-1} \; R_{i-1} \; g_{i-1} \\
&& + \beta (- g_i + u_i )\; \; \mbox{in} \;  \Omega \times
  (0,T) \; , \nonumber \\
  && g_{i,0}(x) = g_i(x,0) \; \; \mbox{on} \; \Omega \; , \\
  && i = 1, \ldots, m \; , \nonumber 
\end{eqnarray}
where $m$ is the number of equations and $i$ is the index of each 
component.
The unknown mobile concentrations $ u_i = u_i(x,t) $ are considered in 
$ \Omega \times (0,T) \subset \R^n \times \R^+$, 
where $n$ is the spatial dimension. 
The unknown immobile concentrations $ g_i = g_i(x,t) $ are considered in 
$ \Omega \times (0,T) \subset \R^n \times \R^+$, 
where $n$ is the spatial dimension. 
The retardation factors $R_i$ are constant and $R_i \ge 0$. 
The kinetic part is given by the factors $\lambda_i$.
They are constant and $\lambda_i \ge 0$. 
For the initialization of the kinetic part, we set $\lambda_0 = 0$.
The kinetic part is linear and irreversible, so the successors 
have only one predecessor. The initial conditions are given for each
component $i$ as constants or linear impulses.
For the boundary conditions we have trivial inflow and outflow conditions
with $u_i = 0$ at the inflow boundary. 
The transport part is given by the velocity ${\bf v} \in \R^n$ and 
is piecewise constant, see \cite{gei_diss03} and \cite{gei_r3t04}.
The exchange between the mobile and immobile part is given 
by $\beta$.

The outline of the paper is as follows.
The splitting-methods are discussed in Section \ref{oper}.
In Section \ref{num}, we present the numerical experiments
and the benefits of the higher order splitting methods.
Finally, we discuss future work in the area of iterative
and non-iterative splitting methods.

\section{Operator splitting methods}
\label{oper}

 We focus our attention on the case of two linear operators (i.e., we consider the Cauchy problem):

\begin{eqnarray}
\label{Cauchy}
&& \frac{\partial c(t)}{\partial t} = A
c (t) \; + \; B c (t), \;
\mbox{with} \; t \in [0,T], \; c(0) = c_0,
\end{eqnarray}
whereby the initial function $c_0$ is given and $A$ and $B$ are assumed to be bounded linear operators in the Banach-space $\mathbf{X}$
with $A, B : \mathbf{X} \rightarrow \mathbf{X}$. In realistic applications the operators correspond to physical operators such as convection and diffusion operators. We consider the following operator splitting schemes:\\

\subsection{Iterative Operator Splitting}

Iterative splitting with respect to one operator
\begin{eqnarray}
 \label{iter_wave}
&& \frac{\partial c_i(t)}{\partial t} = A
c_i(t) \; + \; B c_{i-1}(t), \;
\mbox{with} \; \; c_i(t^n) = c^{n}, i = 1, 2, \ldots, m \;
\end{eqnarray}

\begin{thm}
Let us consider the abstract Cauchy problem given in (\ref{Cauchy}).
Then, we the one-side iterative operator splitting method (\ref{iter_wave})
has the following accuracy:
\begin{eqnarray}
&& || (S_i - \exp((A+B) \tau)|| \le C \tau^i,
\end{eqnarray}
where $S_i$ is the approximated solution for the i-th iterative step
and $C$ is a constant that can be chosen uniformly on bounded time
intervals.
\end{thm}

\begin{proof}

The proof is done for $i= 1, 2, \ldots$
and with the consistency error the $e_i(\tau) = c(\tau) - c_{i}(\tau)$ we have :
\begin{eqnarray}
&& c_i(\tau) = \exp(A \tau) c(t^n) \\
&& + \int_{t^n}^{t^{n+1}} \exp(A (t^{n+1} - s)) B \exp(s A) c(t^n) \; ds \nonumber\\
&& + \int_{t^n}^{t^{n+1}} \exp(A (t^{n+1} - s_1)) B \int_{t^n}^{t^{n+1} - s_1} \exp((t^{n+1} - s_1 - s_2) A) B \exp(s_2 A) c(t^n) \; ds_2 \; ds_1 \nonumber \\
&& + \ldots \nonumber \\
&& + \int_{t^n}^{t^{n+1}} \exp(A (t^{n+1} - s_1)) B \int_{t^n}^{t^{n+1} - s_1} \exp((t^{n+1} - s_1- s_2) A) B \exp(s_2 A) c(t^n) \; ds_2 \; ds_1 + \ldots  \nonumber \\
&& + \int_{t^n}^{t^{n+1}} \exp(A (t^{n+1} - s_1) ) B  \nonumber \\
&& \ldots \int_{t^n}^{t^{n+1} - \sum_{j=1}^{i-2} s_j} \exp(( t^{n+1} - \sum_{j=1}^{i-1} s_j) A) B c_{init}(s_i) \; ds_{i}\; ds_{i-1} \ldots \; ds_2 \; ds_1 \nonumber ,
\end{eqnarray}
\begin{eqnarray}
&& c(\tau)  = \exp(A \tau) c(t^n) \\
&& + \int_{t^n}^{t^{n+1}} \exp(A (t^{n+1} - s)) B \exp(s A) c(t^n) \; ds \nonumber\\
&& + \int_{t^n}^{t^{n+1}} \exp(A (t^{n+1} - s_1)) B \int_{t^n}^{t^{n+1} - s_1} \exp((t^{n+1} - s_1 - s_2) A) B \exp(s_2 A) c(t^n) \; ds_2 \; ds_1 \nonumber \\
&& + \ldots \nonumber \\
&& + \int_{t^n}^{t^{n+1}} \exp(A (t^{n+1} - s_1)) B \int_{t^n}^{t^{n+1} - s_1} \exp((t^{n+1} - s_1- s_2) A) B \exp(s_2 A) c(t^n) \; ds_2 \; ds_1 \nonumber \\
&& + \ldots + \int_{t^n}^{t^{n+1}} \exp(A (t^{n+1} - s_1) ) B  \nonumber \\
&& \ldots \int_{t^n}^{t^{n+1} - \sum_{j=1}^{i-1} s_j} \exp(( t^{n+1} - \sum_{j=1}^{i} s_j) A) B \exp((s_{i} (A+B)) c(t^n) \; ds_{i} \ldots \; ds_2 \; ds_1 \nonumber ,
\end{eqnarray}

We obtain:
\begin{eqnarray}
&& || e_i || \le || \exp((A+B) \tau) c(t^n) - S_i(\tau) c_{init}(\tau) || \nonumber \\
&& \le C \tau^{i} \max_{s_i \in [0, \tau]}|| \exp((s_{i} (A+B)) c(t^n) - c_{init}(s_i) || \nonumber \\
&& \le  C \tau^{i} || \exp((\tau (A+B)) c(t^n) - c_{init}(\tau) ||  ,
\end{eqnarray}
where $i$ is the number of iterative steps.

The same idea can be applied to the
even iterative scheme and also for alternating $A$ and $B$.

\end{proof}

\begin{rema}
The accuracy of the initialisation $c_{init}(\tau)$ is important
to conserve or improve the underlying iterative splitting scheme.

Here we have the following initialization schemes:
\begin{eqnarray}
 c_i(\tau) = \exp(A \tau)  c(t^n) \rightarrow  || e_i || \le  C \tau^{i} c(t^n) , \\
 c_i(\tau) = \exp(A \tau) \exp(B \tau) c(t^n) \rightarrow  || e_i || \le  C \tau^{i+1} c(t^n)  .
\end{eqnarray}
\end{rema}

\subsection{Zassenhaus formula (Sequential Splitting)}
\label{extend}

The Zassenhaus formula is an extension to the exponential splitting schemes 
and is given as:
\begin{eqnarray}
\exp((A + B ) t) = \pi_{i=1}^{j} \exp(a_i A t) \exp(b_i B t) \pi_{k=j}^{m} \exp(C_k t^k) + O(t^{m+1}) .
\end{eqnarray}
where $C_j$ is a function of Lie brackets of $A$ and $B$.

\begin{thm}

The initial value problem (\ref{Cauchy}) is solved by
classical exponential splitting schemes.

Then we can embed the Zassenhaus formula and improve the
classical splitting schemes
\begin{eqnarray}
\exp((A + B ) t) = \pi_{i=1}^{j} \exp(a_i A t) \exp(b_i B t) \pi_{k=j}^{m} \exp(C_k t^k) + O(t^{m+1}) .
\end{eqnarray}
where $C_j$ is a function of the Lie brackets of $A$ and $B$.

\end{thm}

\begin{proof}

1.) Lie-Trotter splitting:

For the Lie-Trotter splitting there exists coefficients with respect to the
extension:
\begin{eqnarray}
 \exp((A+B)t) = \exp(A t) \exp(B t) \Pi_{k=2}^\infty \exp(C_k t^k),
\end{eqnarray}
where the coefficients $C_k$ are given in \cite{geiser_3}.

Based on an existing Baker-Campbell-Hausdorff (BCH) formula of 
the Lie-Trotter splitting one can apply the Zassenhaus formula.

2.) Strang Splitting:

A existing BCH formula is given as:
\begin{eqnarray}
  \exp(A t/2) \exp(B t) \exp(A t/2) = \exp(t S_1 + t^3 S_3 + t^5 S_5 + \ldots), \end{eqnarray}
where the coefficients $S_i$ are given as in \cite{hair02}.

There exists a Zassenhaus formula based on the BCH formula:

%See:
%
\begin{eqnarray}
 \exp((A/2+B/2)t) =\Pi_{k=2}^\infty \exp(\tilde{C}_k t^k) \exp(A/2 t) \exp(B/2 t),
\end{eqnarray}
and
\begin{eqnarray}
 \exp((B/2+A/2)t) = \exp(B/2 t) \exp(A/2 t) \Pi_{k=2}^\infty \exp(C_k t^k),
\end{eqnarray}
then there exists a new product:
\begin{eqnarray}
 \Pi_{k=3}^\infty \exp(D_k t^k) = \Pi_{k=2}^\infty \exp(\tilde{C}_k t^k) \Pi_{k=2}^\infty \exp(C_k t^k),
\end{eqnarray}
with one order higher, see also \cite{yosh90}.

\end{proof}

\begin{rema}
In the following, we concentrate on the Lie-Trotter splitting with the embedded 
Zassenhaus formula, given as:
\begin{eqnarray}
E_{Zassen, Comp, 1}(t) = \exp(A t) \exp(B t) , \\
E_{Zassen, Comp, j}(t) = \exp(C_j t^j) , \mbox{for} \; j \in 2 \ldots, i,
\end{eqnarray}
where the sequential Zassenhaus operator
\begin{eqnarray}
\label{zassen_1}
E_{Zassen, i}(t) = \Pi_{j=1}^i E_{Zassen, Comp, j}(t) ,
\end{eqnarray}
is of accuracy $O(t^i)$ and $i$ are the number of Zassenhaus components.

\end{rema}

\subsection{Embedding Zassenhaus expansion to Iterative Splitting schemes}
\label{extend}

In the following we discuss the embedding of the Zassenhaus formula
into the iterative operator splitting schemes.

\begin{thm}

We solve the initial value problem (\ref{Cauchy}). We assume
bounded and constant operators $A$, $B$. 

The initialization process is done with the Zassenhaus formula:
\begin{eqnarray}
c_i(t) = E_{Zassen,i}(t) c_0.
\end{eqnarray}
where $E_{Zassen,i}(t)$ is given in (\ref{zassen_1}).

Further the improved solutions are embedded to the
iterative splitting schemes (\ref{iter_wave}) and we have after
$j$ iterative steps the following result:
\begin{eqnarray}
c_{i+j}(t) = E_{iter, j}(t) E_{Zassen,i}(t) c_0.
\end{eqnarray}
where we can improve the error of the iterative scheme to
$\mathcal{O}(t^{i+j})$.

\end{thm}

\begin{proof}

The solution of the iterative splitting scheme (\ref{iter_wave}) is given as:
\begin{eqnarray}
c_{i+j}(t) = E_{iter, i}(t) c_{init,j}.
\end{eqnarray}
where $S_i(t) = E_{iter, i}(t)$.

The initialization is given with the Zassenhaus formula as:
\begin{eqnarray}
c_{init,j}(t) = E_{Zassen, j}(t) c_0.
\end{eqnarray}
combining both splitting schemes we have the local error:
\begin{eqnarray}
e_{i+j}(t) = || c(t) - c_{i+j}(t) ||  = || c(t) - E_{iter, i}(t)  E_{Zassen, j}(t) c_0 || \le O(t^{i+j}) c_0.
\end{eqnarray}

\end{proof}

\section{Numerical Examples}
\label{num}

In this section, we discuss examples to the usage of the 
embedded Zassenhaus product methods to the iterative splitting schemes.
In the first examples, we demonstrate somewhat artificially how the proposed
Zassenhaus splitting method avoids the splitting error up to a certain
order. The next examples show the solution of partial
differential equation which can be improved by the 
Zassenhaus products.

In the following, we deal with numerical example to verify 
the theoretical results.

\subsection{First Example: Matrix problem}

For another example, consider the matrix equation,
\begin{eqnarray}
\label{equation3}
  u'(t) & =& \left[
               \begin{array}{cc}
                 1 & 2 \\
                 3 & 0 \\
               \end{array}
             \right]u, \,\,\,\,\,\, u(0)=u_0=\left(
                                         \begin{array}{c}
                                           0 \\
                                           1 \\
                                         \end{array}
                                       \right),
\end{eqnarray}
 the exact solution is
\begin{eqnarray}
\label{exact2}
  u(t) & =& 2(e^{3t}-e^{-2t})/5.
\end{eqnarray}
We split the matrix as,

\begin{eqnarray}
\label{equation4}
  \left[
    \begin{array}{cc}
                 1 & 2 \\
                 3 & 0 \\
               \end{array}
             \right]& =& \left[
               \begin{array}{cc}
                 1 & 1 \\
                 1 & 0 \\
               \end{array}
             \right]+\left[
                       \begin{array}{cc}
                         0 & 1 \\
                         2 & 0 \\
                       \end{array}
                     \right]
\end{eqnarray}

The Figure \ref{simple} present the numerical errors between the exact and the
numerical solution.
\begin{figure}[ht]
\begin{center}  
\includegraphics[width=7.0cm,angle=-0]{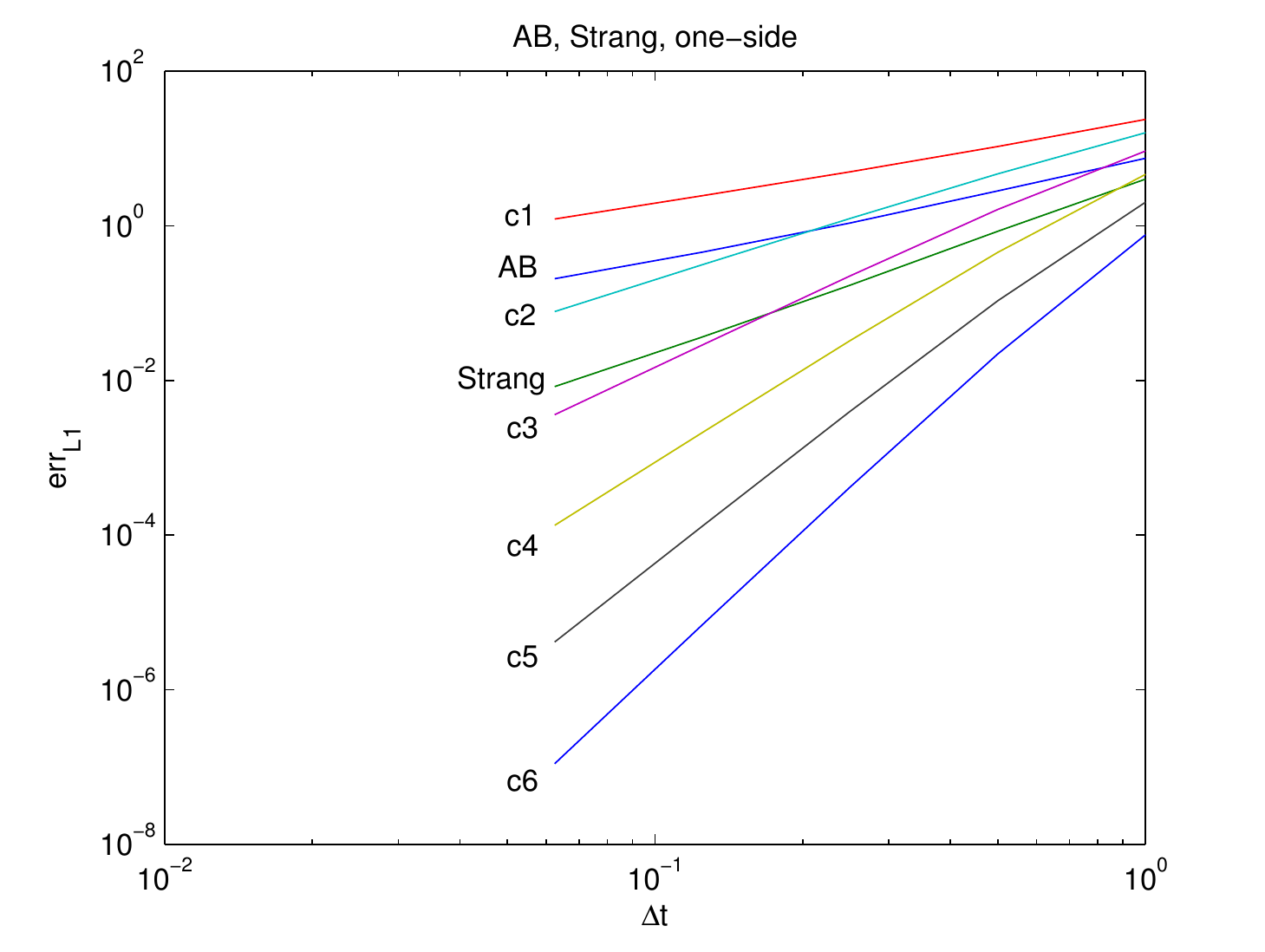} 
\includegraphics[width=7.0cm,angle=-0]{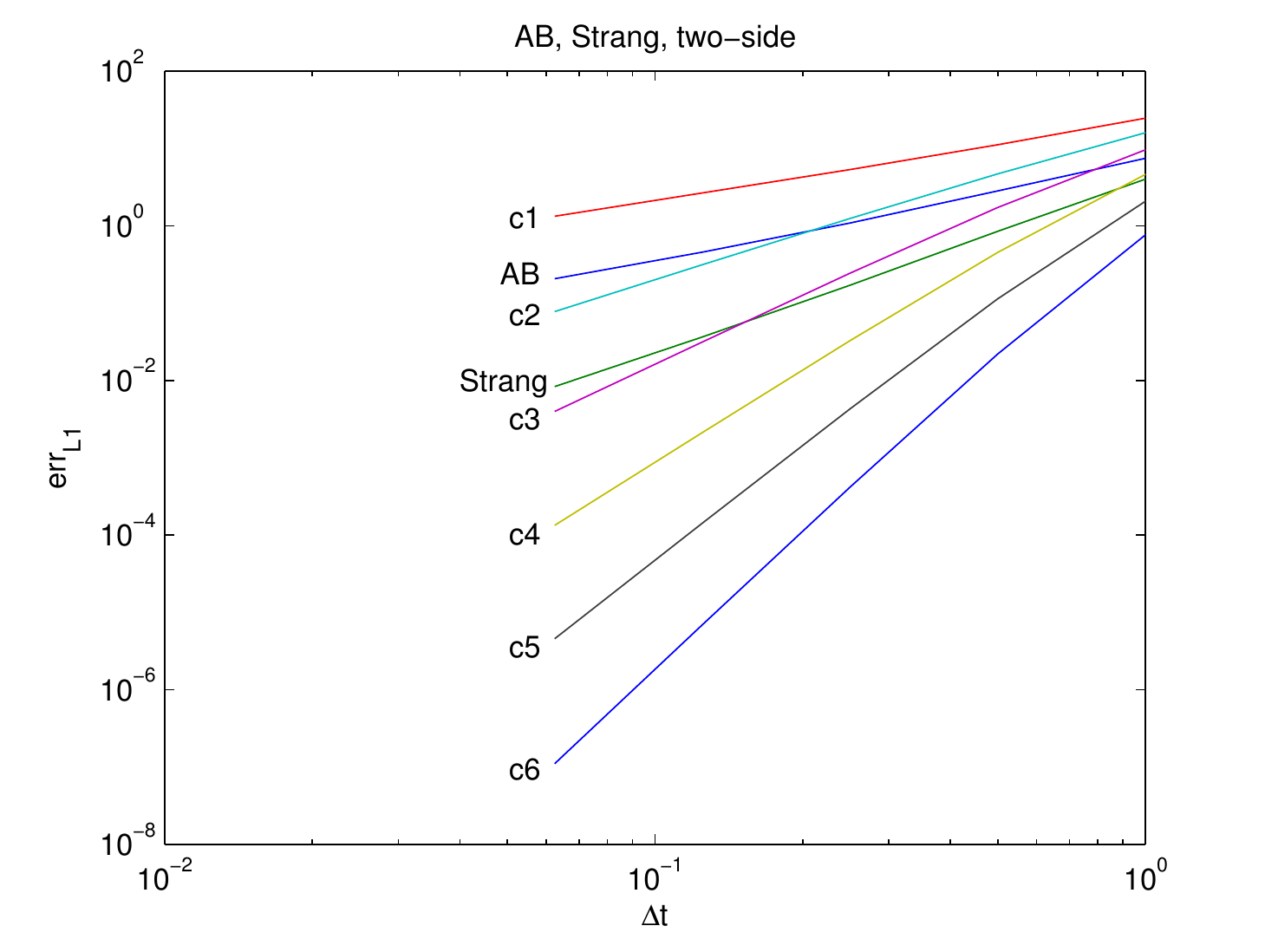} 
\end{center}
\caption{\label{simple} Numerical errors of the standard Splitting scheme and the
iterative schemes with $1, \ldots, 6$ iterative steps.}
\end{figure}

The Figure \ref{simple_comp} present the CPU time of the standard and the iterative splitting schemes.
\begin{figure}[ht]
\begin{center}  
\includegraphics[width=7.0cm,angle=-0]{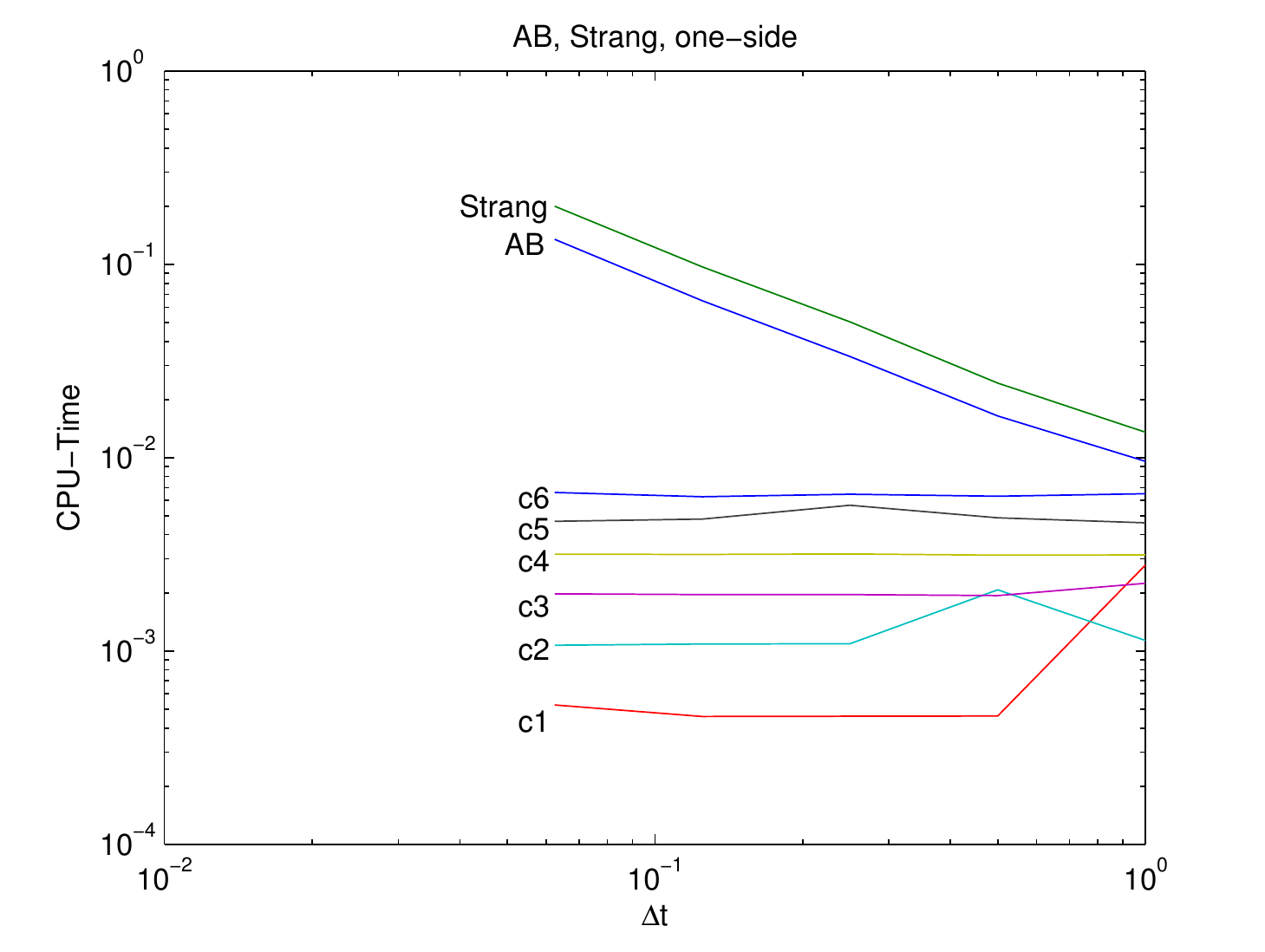} 
\includegraphics[width=7.0cm,angle=-0]{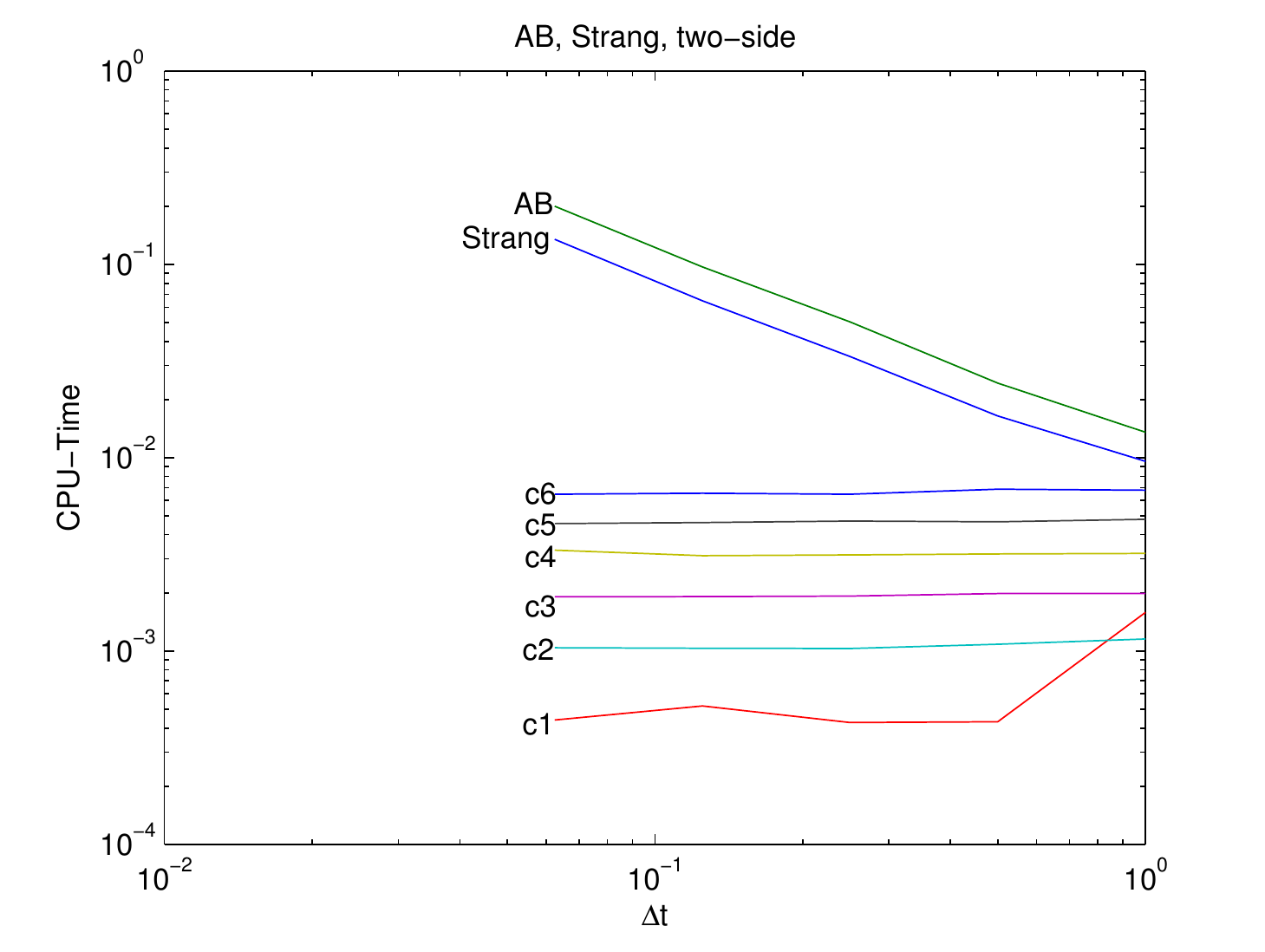} 
\end{center}
\caption{\label{simple_comp} CPU time of the standard Splitting scheme and the
iterative schemes with $1, \ldots, 6$ iterative steps.}
\end{figure}

\begin{rema}
We see that the errors decrease significantly with increasing
order of approximation for the initialization process with the
Zassenhaus splitting. 
Here we have the benefit in the application of the Zassenhaus product
schemes to the standard Lie-Trotter or Strang-Marchuk splitting. 
Similar results are given with the iterative steps $i=3 \ldots, 6$, as expected.
\end{rema}

\subsection{One phase example}

The next example is a simplified real-life problem 
for a multiphase transport-reaction equation.
We deal with mobile and immobile pores in the porous media,
such simulations are given for waste scenarios.

We concentrate on the computational benefits of a fast
computation of the mixed iterative scheme with the
Zassenhaus formula. \\

The one phase equation is given as:
\begin{eqnarray}
\label{mobile1_1}
&& \partial_t c_1  + \nabla \cdot {\bf F} c_1  = - \lambda_1 c_1  ,  \; \mbox{in} \; \Omega \times [0, t] , \\
&& \partial_t c_2  + \nabla \cdot {\bf F} c_2  =  \lambda_1 c_1 - \lambda_2 c_2  ,  \; \mbox{in} \; \Omega \times [0, t] , \\
&& {\bf F}  = {\bf v}  - D \nabla , \\
&& c_1({\bf x}, t) = c _{1,0}({\bf x}),  c_2({\bf x}, t) = c _{2,0}({\bf x}) , \; \mbox{on} \; \Omega , \\
&& c_1({\bf x}, t) = c_{1,1}({\bf x},t) , c_2({\bf x}, t) = c_{2,1}({\bf x},t) , \;  \mbox{on} \; \partial \Omega  \times [0, t] .
\end{eqnarray}
where we have the parameters:
$v= 0.1$, $D=0.01$, $\lambda_1 = \lambda_2 = 0.1$.

In the following we deal with the  finite difference schemes for the
convection and diffusion operators and semidiscretize the equation,
which is given as:
\begin{eqnarray}
\label{mobile1_2}
&& \partial_t {\bf c}  = (A_1 + A_2) {\bf c} ,
\end{eqnarray}
We obtain the two matrices and consider to decouple the diffusion 
and convection part:
\begin{eqnarray}
A_1 & = &  \left(\begin{array}{c c}
 A_{diff} & 0 \\
 0 &  A_{diff}
\end{array}\right) ~\in~\R^{2 I \times 2 I}
\end{eqnarray}
\begin{eqnarray}
A_2 & = &  \left(\begin{array}{c c}
 A_{Conv} & 0 \\
 0 & A_{Conv}
\end{array}\right) + \left(\begin{array}{c c}
 - \Lambda_1 & 0 \\
 \Lambda_1 & - \Lambda_2
\end{array}\right) ~\in~\R^{2 I \times 2 I}
\end{eqnarray}
For the operator $A_1$ and $A_2$ we apply the splitting method,
given in Section \ref{para}.

The submatrices are given in the following:
\begin{eqnarray}
A & = &   \left(\begin{array}{c c}
 A_{diff} & 0 \\
 0 &  A_{diff}
\end{array}\right) +   \left(\begin{array}{c c}
 A_{conv} & 0 \\
 0 &  A_{conv}
\end{array}\right) \nonumber  \\
 & = &  \frac{D}{\Delta x^2}\cdot  \left(\begin{array}{rrrrr}
 -2 & 1 & ~ & ~ & ~ \\
 1 & -2 & 1 & ~ & ~ \\
 ~ & \ddots & \ddots & \ddots & ~ \\
 ~ & ~ & 1 & -2 & 1 \\
 ~ & ~ & ~ & 1 & -2
\end{array}\right) \nonumber \\[6pt]
 & - & \frac{v}{\Delta x}\cdot \left(\begin{array}{rrrrr}
 1 & ~ & ~ & ~ & ~ \\
 -1 & 1 & ~ & ~ & ~ \\
 ~ & \ddots & \ddots & ~ & ~ \\
 ~ & ~ & -1 & 1 & ~ \\
 ~ & ~ & ~ & -1 & 1
\end{array}\right)~\in~\R^{I \times I}
\end{eqnarray}
where $I$ is the number of spatial points and $\Delta x$ is the spatial step size.
\begin{eqnarray}
\Lambda_1 & = &    \left(\begin{array}{rrrrr}
 \lambda_1 & 0 & ~ & ~ & ~ \\
        0 &  \lambda_1 & 0 & ~ & ~ \\
 ~ & \ddots & \ddots & \ddots & ~ \\
 ~ & ~ & 0 & \lambda_1 & 0 \\
 ~ & ~ & ~ & 0 & \lambda_1
\end{array}\right) ~\in~\R^{I \times I}
\end{eqnarray}
\begin{eqnarray}
\Lambda_2 & = &    \left(\begin{array}{rrrrr}
 \lambda_2 & 0 & ~ & ~ & ~ \\
        0 &  \lambda_2 & 0 & ~ & ~ \\
 ~ & \ddots & \ddots & \ddots & ~ \\
 ~ & ~ & 0 & \lambda_2 & 0 \\
 ~ & ~ & ~ & 0 & \lambda_2
\end{array}\right) ~\in~\R^{I \times I}
\end{eqnarray}

We have the following results:

We have the spatial step size $\Delta x = 0.1$.

The Figure \ref{one_phase_1} present the numerical errors between the exact and the
numerical solution.
\begin{figure}[ht]
\begin{center}  
\includegraphics[width=7.0cm,angle=-0]{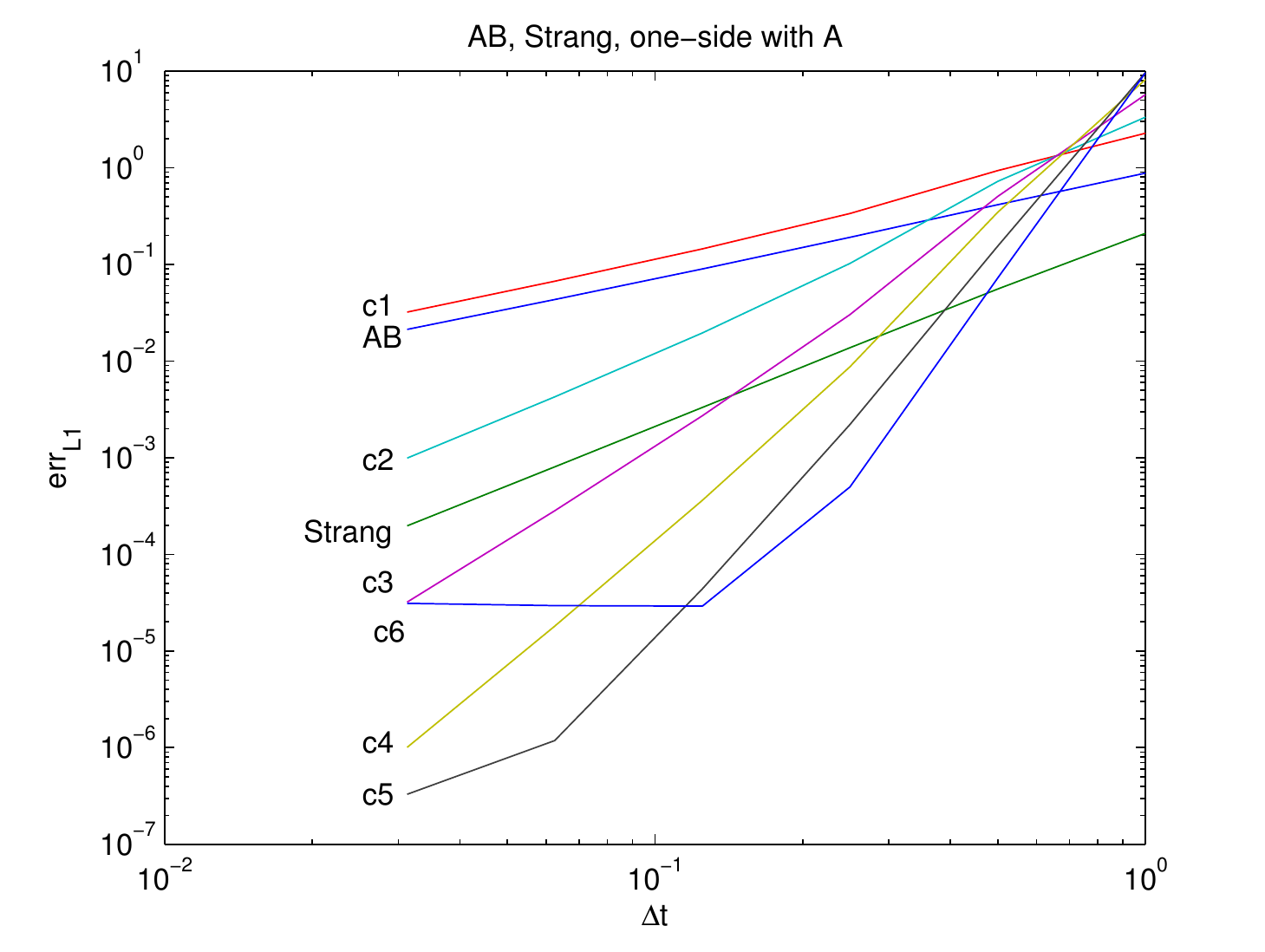} 
\includegraphics[width=7.0cm,angle=-0]{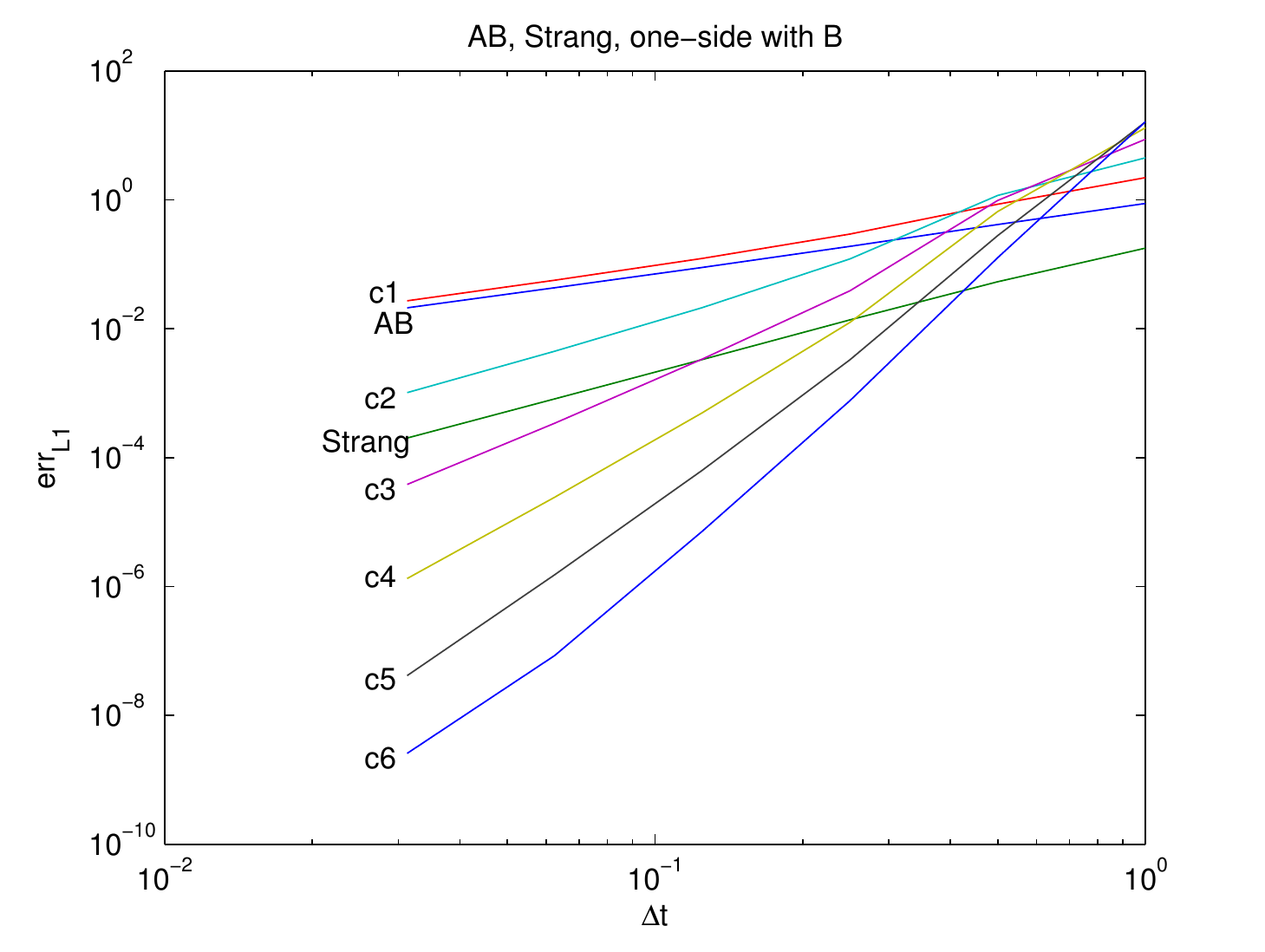} 
\includegraphics[width=7.0cm,angle=-0]{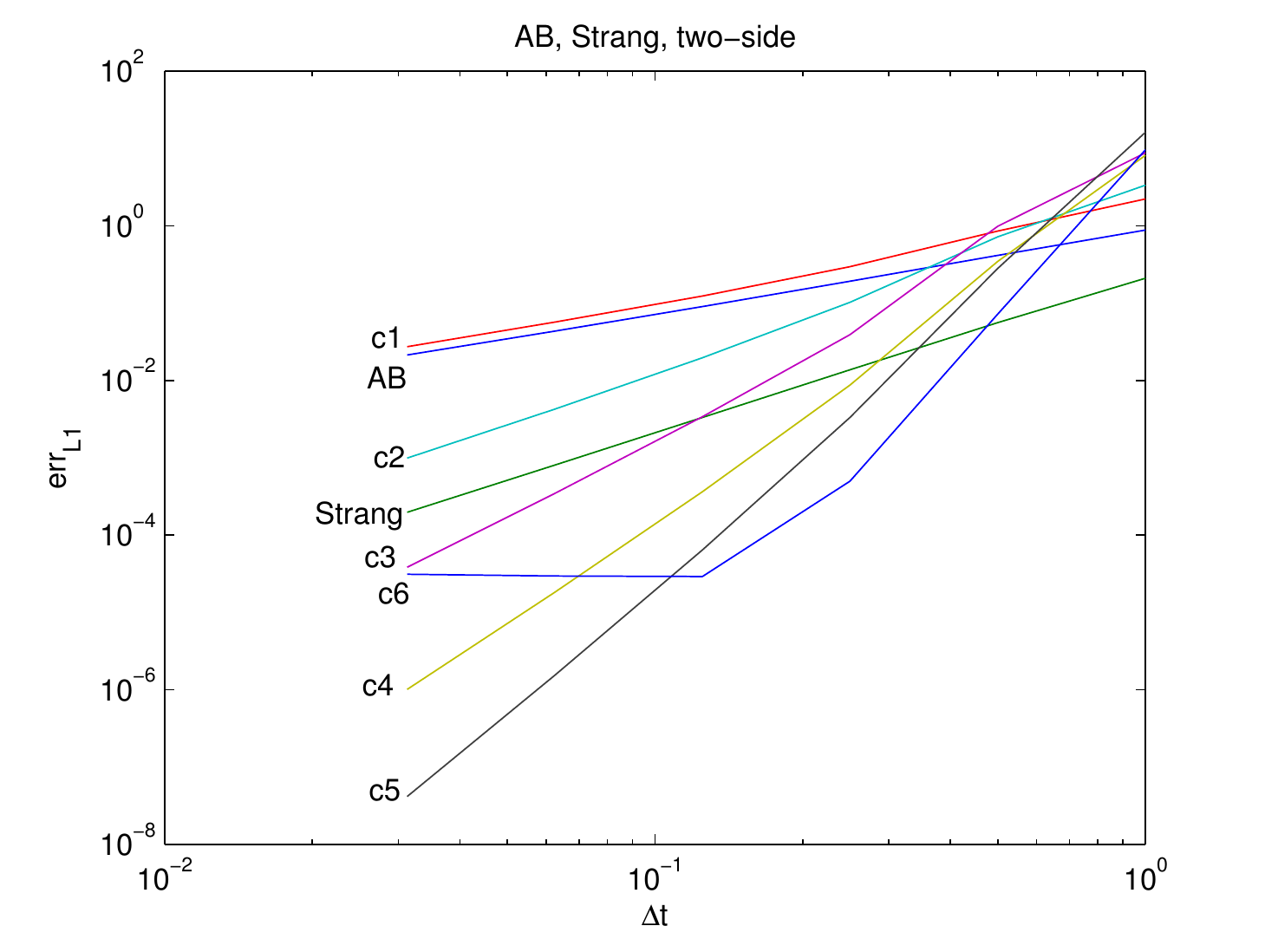} 
\end{center}
\caption{\label{one_phase_1} Numerical errors of the standard Splitting scheme and the
iterative schemes with $1, \ldots, 6$ iterative steps.}
\end{figure}

The Figure \ref{one_phase_CPU} present the CPU time of the standard and the iterative splitting schemes.
\begin{figure}[ht]
\begin{center}  
\includegraphics[width=7.0cm,angle=-0]{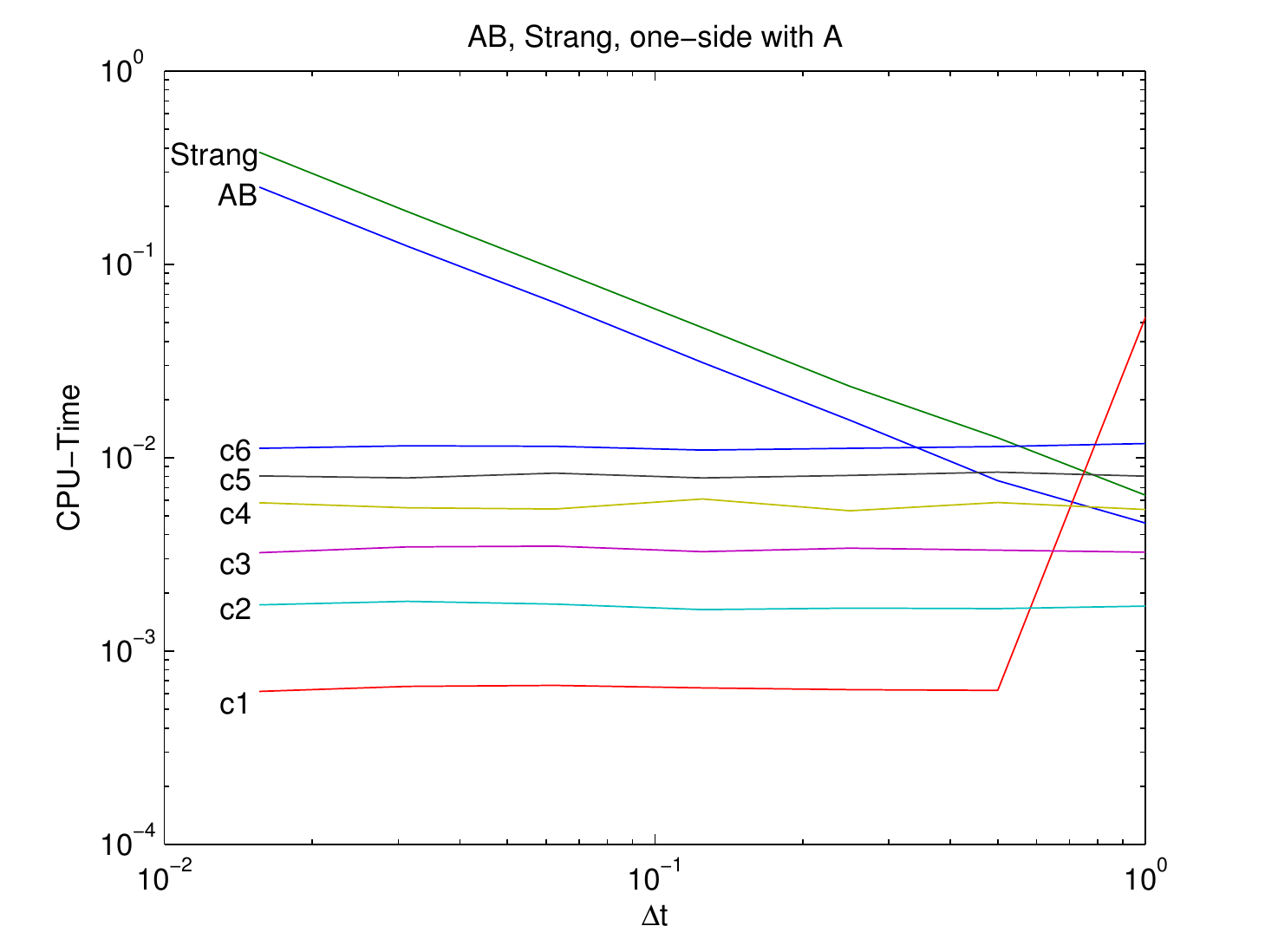}
\includegraphics[width=7.0cm,angle=-0]{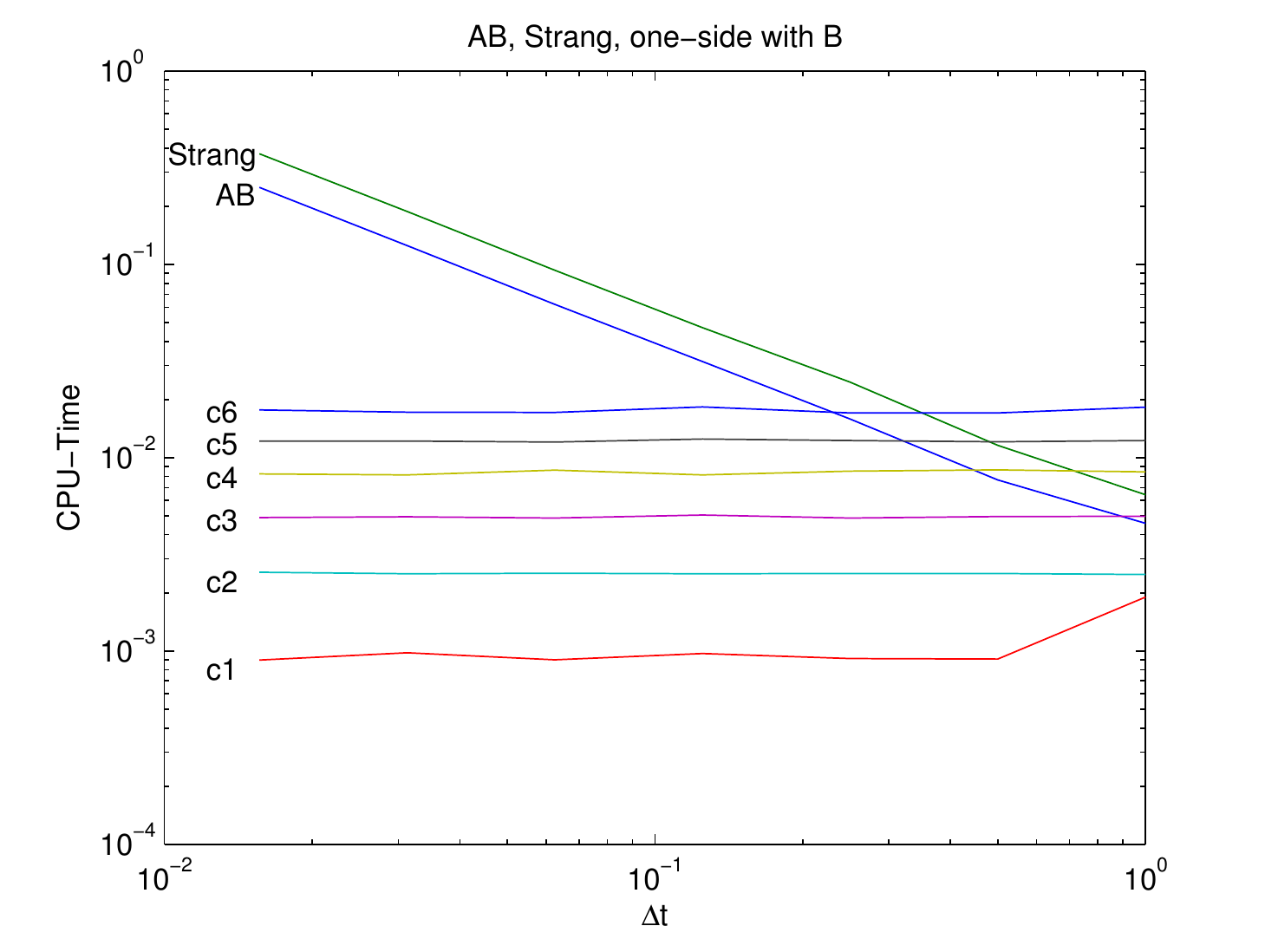}
\includegraphics[width=7.0cm,angle=-0]{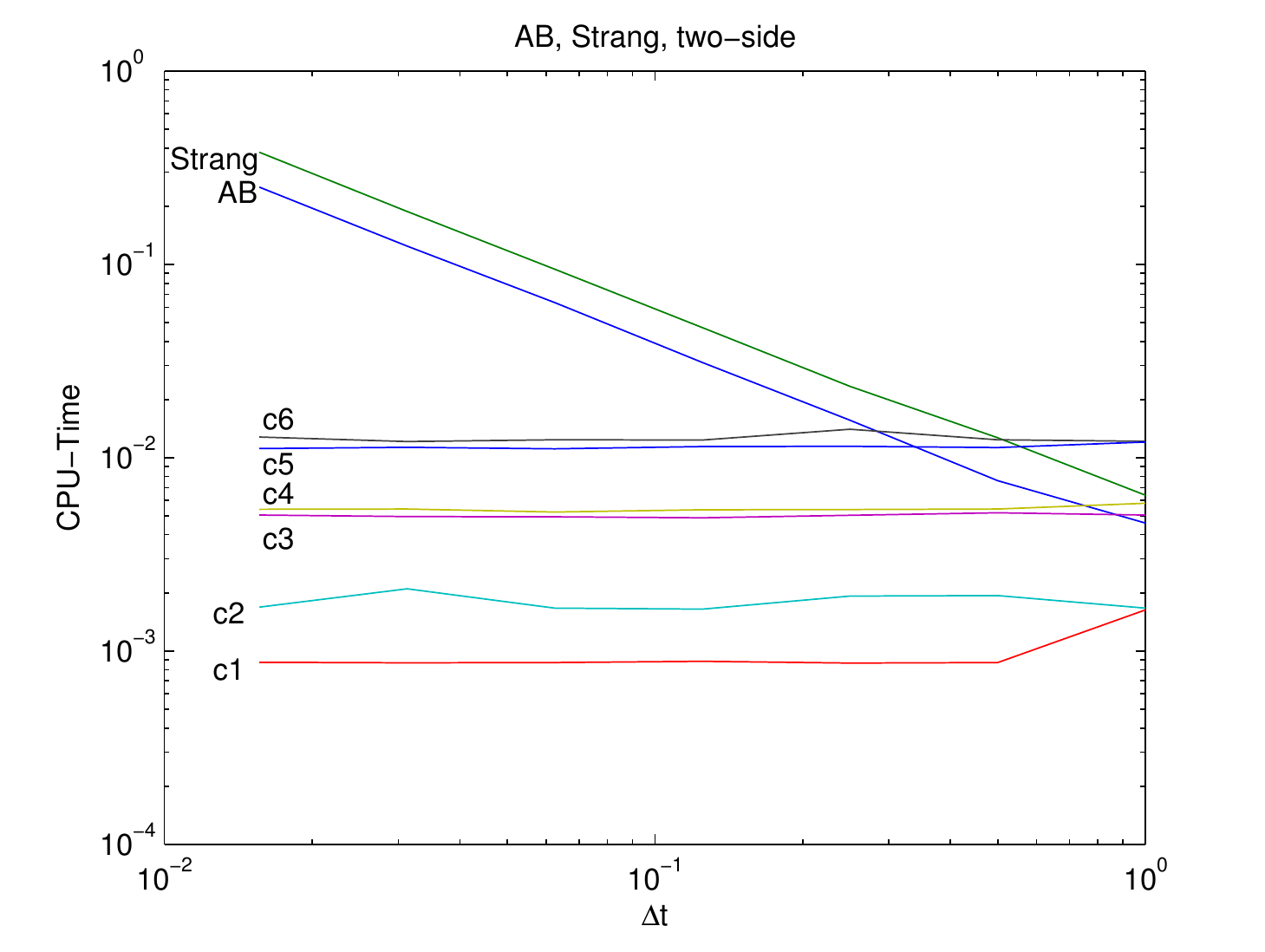} 
\end{center}
\caption{\label{one_phase_CPU} CPU time of the standard Splitting scheme and the
iterative schemes with $1, \ldots, 6$ iterative steps.}
\end{figure}

\begin{rema}
For the iterative schemes with embedded Zassenhaus products, we can reach 
faster and more improved results.
With $4-5$ iterative steps we obtain more accurate results as we did
for the expensive standard schemes.
With one-side iterative schemes we reach the best convergence results.
\end{rema}

\section{Conclusions and Discussions }
\label{conc}
 In this work, we have presented a  novel splitting scheme
combing the ideas of iterative and sequential schemes.
Here the idea to decouple the expensive computation of 
only matrix exponential based schemes to simpler embedded 
Zassenhaus schemes, which have their benefits of 
less computational time, while the commutator can be computed very cheap.
On the other hand simple linear iterative steps can be done very cheap
and accelerated the solvers.
The error analysis presented stable methods for the higher order schemes.
In the applications, we could show the speedup with the Zassenhaus 
enhanced methods.
In future we concentrate on linear and nonlinear matrix dependent 
scheme, that switches between no-iterative and iterative schemes 
based on Zassenhaus products.

%%%%%%%%%%%%%%%%%%%%%%%%%%%%%%%%%%%%%%%%%
%%             REFERENCES              %%
%%%%%%%%%%%%%%%%%%%%%%%%%%%%%%%%%%%%%%%%%

\bibliographystyle{plain}

\end{document}